\documentclass[12pt,a4paper,oneside,leqno]{article}     
\usepackage[utf8]{inputenc}
\usepackage{t1enc}
\usepackage[english]{babel}
\usepackage{amssymb}
\usepackage{amsmath}        
\usepackage{amsthm}
\usepackage{graphicx}
\usepackage{fancyhdr}
\usepackage{wrapfig}
\usepackage{euscript}       
\usepackage{indentfirst}

\newtheorem{all}{Proposition}
\newtheorem{lem}[all]{Lemma}

\newtheorem{kov}[all]{Corollary}

\let\oo\H

\numberwithin{equation}{section}
\def\J{\mathbf{J}}
\def\i{\mathbf{i}}
\def\j {\mathbf{j}}
\def\V{\mathcal{V}}
\def\v{\mathbf{v}}
\def\h{\mathbf{h}}      
\def\X{{\widetilde{X}}}
\def\Y{{\widetilde{Y}}}     
\def\Z{{\widetilde{Z}}}    
\def\L{\mathcal{L}_{X^{\mathsf c}}}         
\def\La{\widetilde{\mathcal{L}}_{X^{\mathsf c}}}
\def\l{\mathcal{L}}
\def\la{\widetilde{\mathcal{L}}}
\def\H{\mathcal{H}}
\def\R{\mathbf{R}}
\newcommand{\ala}[1]{\underline{#1}}
\newcommand{\kal}[1]{\widehat{#1}}
\newcommand{\vm}{\mathfrak{X}}
\newcommand{\vl}{^\mathsf{v}}
\newcommand{\tl}{^\mathsf{c}}
\newcommand{\hl}{^\mathsf{h}}
\newcommand{\splt}{\overset\circ{T}}
\newcommand{\splp}{\overset\circ{\pi}}
\newcommand{\Sec}{\mathrm{Sec}}
\newcommand{\tenz}{\EuScript{T}}
\newcommand{\kr}{\mathrm{Ker}}
\newcommand{\im}{\mathrm{Im}}

\newcommand{\dvr}{\mathop{\mathrm{div}}}
\newcommand{\inv}{^{-1}}
\newcommand\ujsor{      
\par\noindent}
\newcommand{\eqnbetu}[2]{ 
\addtocounter{equation}1
\xdef #1{\arabic{section}.\arabic{equation}}
\tag{#1#2}}
\newcommand{\rel}[3]{       
#1\raisebox{-2pt}[0pt][-7pt]{$\begin{array}{c}\sim\\[-10pt]\scriptstyle{#2}\end{array}$} #3}
\newtheoremstyle{szoveg} 
{\topsep}                    
{\topsep}                    
{}                   
{}                           
{\bfseries}                   
{}                          
{.5em}                       
{\thmnumber{#2}\thmnote{ #3}}  
\theoremstyle{szoveg}
\newtheorem{szov}{nemlatszik}[section] 
\title{Conformal vector fields on Finsler manifolds}
\author{József Szilasi  \qquad  Anna Tóth}
\date{}
\begin{document}
{\setlength\arraycolsep{.13889em}
 \maketitle
\begin{abstract}
Applying concepts and tools from classical tangent bundle geometry and using the apparatus of the calculus along the tangent bundle
projection (`pull-back formalism'), first we enrich the known lists of the characterizations of affine vector fields on a spray manifold and
conformal vector fields on a Finsler manifold. Second, we deduce consequences on vector fields on the underlying manifold of a Finsler
structure having one or two of the mentioned geometric properties.
\end{abstract}



\section*{Introduction}
The theory of `geometrical' -- projective, affine, conformal, isometric -- vector fields on a Finsler manifold has a vast literature,
mainly from the period dominated technically by the classical tensor calculus, visually, `the debauch of indices'. Chapter VIII of K.~Yano's
book `The theory of Lie derivatives and its applications' presents a survey of the main achievements from the beginning of the 20th century
to 1957. A good overview of the developments of the next decades can be found in R.~B.~Misra's paper \cite{misra}, written in 1981,
revised and updated in 1993. It is important to note that in a 2-part paper, see \cite{matsumoto1},\cite {matsumoto2}, M.~Matsumoto clarified and improved some
results of Yano in the framework of his theory of Finsler connections.

From the (relatively) modern, but partly tensor calculus based literature the works of H.~Akbar-Zadeh
\cite{az1},\cite{az2}, J.~Grifone \cite{grifon2},\cite{grifon3} and R.~L.~Lovas \cite{lovas} are worth mentioning. Grifone applies systematically the
\mbox{`$\tau_{TM}\colon TTM \to TM$ formalism'}, combining with the Frölicher--Nijenhuis calculus of vector-valued forms; Lovas formulates and proves
his results in terms of the `pull-back formalism $\splp\colon \overset\circ TM \times_M TM \to \overset\circ TM$'. Our paper is a continuation of
both Grifone's and Lovas's works. Although we are going to develop the greater part of the theory in terms of the pull-back bundle,
the concepts and techniques of the tangent bundle geometry, including the vertical calculus on $TM$, also play an eminent role in our
considerations. To make the paper more readable, in section 1 we summarize in a coherent way the various concepts and tools which
will be indispensable in the following.

We apply two types of a Lie derivative operator: beside the classical Lie derivative operator
$\l_\xi$ on $TM$ ($\xi \in \vm(TM)$) we need a further operator, denoted by $\la_\xi$, which acts on the tensor algebra of the
$C^{\infty}(TM)$-module of the sections of the vector bundle $\pi \colon TM \times_M TM \to TM$ (\mbox{or of the bundle $\splp$}).
To assure
the validity of the crucial identity $[\la_\xi,\la_\eta]=\la_{[\xi,\eta]}$ in case of the `new' operator, we are forced to differentiate with
respect to \emph{projectable} vector fields on $TM$. In section 2 some basic properties of the operator $\la_\xi$ are established.

The affine and projective properties of a Finsler manifold depend only on its canonical spray, so it is natural to examine affine and
projective vector fields in the (virtual) generality of spray manifolds. A vector field $X$ on a manifold $M$ is said to be an affine
vector field or a Lie symmetry for a spray $S\colon TM \to TTM$ if $S$ is invariant under the flow of the complete lift $X\tl$ of $X$,
that is, if $\L S=[X\tl,S]=0$. In Lovas's paper \cite{lovas} various equivalents of this property are established. In section 3 we enrich his list
with some new items, which will be technically useful in the next section.

By a conformal vector field on a Finsler manifold $(M,F)$ we mean a vector field $X$ on $M$ satisfying
\[
 \La g=\varphi\, g,
\]
where $g$ is the metrical tensor of the Finsler manifold (the vertical Hessian of the energy function $E=\frac{1}{2}F^2$)
and $\varphi$ is a function, defined and continuous on $TM$, smooth on the deleted bundle $\overset\circ TM$.
It turns out at once that $\varphi$ has
to be fibrewise constant, i.e., of the form $\varphi=f\circ \tau$, where $f$ is a smooth function on $M$ and $\tau$ is the tangent bundle projection.
Homothetic and isometric (or Killing) vector fields are the particular cases for which $\varphi$ is a constant function, resp.\ identically zero.
In section 4 we present further characterizations of conformal vector fields on a Finsler manifold (Proposition \ref{conformal}), one of them has already
been proposed by Grifone in \cite{grifon3}. We show that if a vector field $X\in \vm(M)$ is both affine and conformal on a Finsler manifold
$(M,F)$, then $X\tl$ is a conformal vector field for the Sasaki extension of the metric tensor of $(M,F)$ (Proposition \ref{confsasaki}).

At this stage, the following `expectable', but non-trivial conclusions may be deduced fairly easily:
\begin{enumerate}
\item[(a)] Homothetic vector fields on a Finsler manifold are affine vector fields (Proposition \ref{homaff}).
\item[(b)] If a vector field on a Finsler manifold is both projective and conformal, then it is a homothetic vector field (Proposition \ref{projconf}).
\item[(c)] If a vector field preserves the Dazord volume form of a Finsler manifold and it is also projective, then it is an affine vector
field \mbox{(Proposition \ref{vol}, (i))}.
\item[(d)] If a vector field is both volume-preserving (in the above sense) and conformal, then it is a Killing field (Proposition \ref{vol}, (ii)).
\end{enumerate}
\section{Basic setup}
\begin{szov}[Generalities]\label{generalities} 
Most of our basic notations and conventions will be the same as in \cite{bacso}, see also \cite{setting}.
However, for the reader's convenience, we present here a short review on the most essential things.
\ujsor (\textbf{a})\label{manifold} By a manifold we mean a finite dimensional smooth manifold whose underlying topological space is Hausdorff,
 second countable and connected.
In what follows, $M$ will be an $n$-dimensional manifold, where \mbox{$n\geq2$}. Let \mbox{$k\in\mathbb N\cup\{\infty\}$}. We denote by $C^k(M)$
the set of $k$-times continuously differentiable real-valued functions on $M$, with the convention that $C^0(M)$ is the set of the continuous functions on $M$.
In particular, $C^\infty(M)$ is the real algebra of smooth functions on $M$.
\ujsor (\textbf{b}) The tangent space of $M$ at a point \mbox{$p\in M$} is denoted by $T_pM$; \mbox{$TM:=\bigcup_{p\in M}T_pM$}.
The tangent bundle of $M$ is the triplet $(TM,\tau,M)$,
where the tangent bundle projection $\tau$ is defined by \mbox{$\tau(v):=p$} if \mbox{$v\in T_pM$}. Instead of $(TM,\tau,M)$ we usually write
 \mbox{$\tau\colon TM\rightarrow M$} or simply $\tau$.
Similarly, the tangent bundle of $TM$ is $(TTM, \tau_{TM},TM)$ or \mbox{$\tau_{TM}\colon TTM\rightarrow TM \mbox { or } \tau_{TM}$.}
In general, we prefer to denote a bundle by the same symbol as we use for its projection.

A \emph{vector field} on $M$ is a smooth section of the tangent bundle \mbox{$\tau\colon TM\to M$}.
The vector fields on $M$ form a $C^{\infty} (M)$-module which will be denoted by
$\mathfrak{X}(M)$. The \emph{zero vector field} $o$ on $M$ is defined by
\[
 p\in M \mapsto o(p):=0_p:= \mbox{ \emph{the zero vector} in } T_pM.
\]
The \emph{deleted bundle} for $\tau$ is the fibre bundle \mbox{$\overset\circ{\tau}\colon\overset\circ TM\to M$}, where
\mbox{$\overset\circ TM:=TM\setminus o(M),\ \overset\circ\tau:=\tau\upharpoonright \overset\circ TM$}.

\ujsor (\textbf{c}) If $\varphi:M\rightarrow N$ is a smooth mapping between smooth manifolds, then we denote its derivative by $\varphi_*$,
which is a fibrewise linear smooth mapping of $TM$ into $TN$.
Two vector fields $X\in\vm(M)$ and $Y\in\vm(N)$ are \emph{$\varphi$-related} if $\varphi_*\circ X=Y\circ\varphi$; then we write $\rel X\varphi Y$.
A vector field $\xi$ on $TM$ is said to be \emph{projectable} if there exists a vector field $X$ on $M$ such that $\rel\xi\tau X$.

\ujsor (\textbf{d}) The classical graded derivations of the graded algebra \mbox{$\Omega(M):=\bigoplus_{k=0}^n\Omega^k(M)$} of the differential forms on $M$ are
\begin{center}
\begin{tabular}{l}
the Lie derivative $\l_X$ $(X\in\vm(M)$),\\
the substitution operator $i_X$ $(X\in\vm(M))$,\\
the exterior derivative $d$,
\end{tabular}
\end{center}
related by H.~Cartan's `magic' formula
\begin{equation}\label{magic}
\l_X=i_X\circ d+d\circ i_X.
\end{equation}
\end{szov}
\begin{szov}[Canonical constructions and objects]\label{Canob}\
\ujsor (\textbf{a}) By the \emph{vertical lift} of a smooth function $f$ on $M$ we mean the function
\[f\vl:=f\circ\tau\in C^\infty(TM);\]
the \emph{complete lift} of $f$ is the function $f\tl\in C^\infty(TM)$ given by
\[f\tl(v):=v(f),\ v\in TM.\]

\ujsor (\textbf{b}) A vector field $\xi$ on $TM$ is \emph{vertical} if $\rel \xi\tau o$.
The vertical vector fields form a $C^\infty(TM)$-module $\vm\vl(TM)$, which is also a subalgebra of the Lie algebra $\vm(TM)$. The \emph{Liouville vector field}
 on $TM$ is the unique vertical vector field $C\in\vm\vl(TM)$ such that
\begin{equation}\label{Cff}
Cf\tl=f\tl \textrm{ \emph{for all} } f\in C^\infty(M).
\end{equation}
The \emph{vertical lift of a vector field} $X$ on $M$ is the unique vertical vector field $X\vl\in\vm\vl(TM)$ satisfying
\begin{equation}\label{Xvfv}
X\vl f\tl=(Xf)\vl\textrm{ \emph{for all} }f\in C^\infty(M);
\end{equation}
the \emph{complete lift} $X\tl\in\vm(TM)$ of $X$ is characterized by
\begin{equation}\label{Xcfc}
X\tl f\tl=(Xf)\tl,\ f\in C^\infty(M)
\end{equation}
(see \cite{yano2}, Ch. I.3). Then we have
\begin{equation}\label{Xcfv}
X\tl f\vl=(Xf)\vl,\ f\in C^\infty(M).
\end{equation}
Both $X\vl$ and $X\tl$ are projectable: $\rel {X\vl}\tau o$, $\rel {X\tl}\tau X$. Lie brackets involving vertical and complete lifts satisfy the rules
\begin{align}\eqnbetu{\lielift}{a-c}    
&[X\vl,Y\vl]=0,\quad [X\tl,Y\vl]=[X,Y]\vl,\quad [X\tl,Y\tl]=[X,Y]\tl,\\\eqnbetu{\lieC}{a-b}
&[C,X\vl]=-X\vl,\quad [C,X\tl]=0.
\end{align}

\ujsor (\textbf{c}) Let
\begin{align*}
TM\times_MTM:&=\big\{(u,v)\in TM\times TM\big|\ \tau(u)=\tau(v)\big\},\\
\splt M\times_MTM:&=\big\{(u,v)\in \splt M\times TM\big|\ \overset\circ\tau(u)=\tau(v)\big\}.
\end{align*}
If
\[
\pi:=\mathrm{pr}_1\upharpoonright TM\times_MTM,\quad \splp:=\mathrm{pr}_1\upharpoonright \splt M\times_MTM,
\]
then both $\pi$ and $\splp$ are vector bundles over $TM$ and $\splt M$, resp., with fibres
\[
\{u\}\times T_{\tau(u)}M\cong T_{\tau(u)}M;\quad u\in TM, \textrm{ resp. } u\in\splt M.
\]
We denote by $\Sec(\pi)$ and $\Sec(\splp)$ the $C^\infty(TM)$-, resp.\ $C^\infty(\splt M)$-module of the sections of these bundles.
 A typical section in $\Sec(\pi)$ is of the form
\[
\X:v\in TM\longmapsto (v,\ala X(v))\in TM\times_MTM,
\]
where $\underline X:TM\rightarrow TM$ is a smooth mapping such that $\tau\circ\ala X=\tau$. $\ala X$ is called the \emph{principal part} of $\X$.
We have a \emph{canonical section} in $\Sec(\pi)$, denoted by $\delta$, whose principal part is the identity mapping of $TM$.
Every vector field $X$ on $M$ yields a section $\kal X$ in $\Sec(\pi)$, called a \emph{basic section}, whose principal part is $X\circ\tau$.
Locally, the $C^\infty(TM)$-module $\Sec(\pi)$ is generated by the basic sections.

We denote by $\tenz^k_l(\pi)$ the $C^\infty(TM)$-module of the type $(k,l)$ tensors over the module $\Sec(\pi)$; the meaning of $\tenz^k_l(\splp)$ is analogous.

\ujsor (\textbf{d}) We have a canonical $C^\infty(TM)$-linear injection $\i\colon\Sec(\pi)\rightarrow\vm(TM)$ given on the basic sections by
\begin{equation}\label{ix}
\i(\kal X):=X\vl,\ X\in\vm(M),
\end{equation}
and a canonical $C^\infty(TM)$-linear surjection $\j\colon \mathfrak{X}(TM)\to \Sec(\pi)$ such that
\begin{equation}\label{jx}
\j(X\vl):=0,\ \j(X\tl):=\kal X.
\end{equation}
Then $\im(\i)=\kr(\j)=\vm\vl(TM)$.
The mapping $\J:=\i\circ\j$ is said to be the \emph{vertical endomorphism} of $\vm(TM)$. It follows immediately that
\[
\im(\J)=\kr(\J)=\vm\vl(TM),\ \J^2=0.
\]
Due to their $C^\infty(TM)$-linearity, $\i$, $\j$ and $\J$ have a natural pointwise interpretation.
\end{szov}
\begin{szov}[Some vertical calculus]\
\ujsor (\textbf{a}) We define the \emph{vertical differential} $\nabla\vl F$ of a function $F\in C^\infty(TM)$ as a $1$-form in $\tenz^0_1(\pi)$ given by
\begin{equation}\label{Dvf}
\nabla\vl F(\X):=\nabla\vl_\X F:=(\i\X)F,\ \X\in\Sec(\pi).
\end{equation}

The vertical differential $\nabla\vl\Y$ of a section $\Y\in\Sec(\pi)$ is the type $(1,1)$ tensor in $\tenz^1_1(\pi)$ defined by
\begin{equation}\label{DvY}
\left\{\begin{array}{l}
\nabla\vl\Y(\X):=\nabla\vl_\X\Y:=\j[\i\X,\eta],\\
\eta\in\vm(TM),\ \j(\eta)=\Y.
\end{array}\right.
\end{equation}
(It is easy to check that $\nabla\vl_\X\Y$ does not depend on the choice of $\eta$ satisfying $\j(\eta)=\Y$.)

By the standard technique, to make sure that Leibniz's rule holds, the operators $\nabla\vl_\X$ may be extended to tensor derivations of the full tensor
algebra of $\Sec(\pi)$.
\ujsor (\textbf{b}) Next we consider the graded algebra $\Omega(TM)$ of the differential forms on $TM$, and we define an operator
\[d_\J:\Omega(TM)\longrightarrow\Omega(TM)\]
by the rules
\begin{equation}\label{dJ}
d_\J F:=dF\circ\J,\quad d_\J dF:=-d\,d_\J F;\quad F\in C^\infty(TM).
\end{equation}
Then $d_\J$ is a graded derivation of degree $1$ of $\Omega(TM)$, called the \emph{vertical differentiation on $TM$}.
 We have (and we shall need) the following important relation:
\begin{equation}\label{dJLC}
d_\J\circ\l_C-\l_C\circ d_\J=d_\J.
\end{equation}
For details, we refer to the book \cite{leon}. We mention that $\nabla\vl$ and $d_\J$, at the level of functions, are related by
\[d_\J F=\nabla\vl F\circ\j,\ F\in C^\infty(TM).\]
\ujsor (\textbf{c}) Let $K$ be a type $(1,1)$ tensor on $TM$, interpreted as an endomorphism of the $C^\infty(TM)$-module $\vm(TM)$.
It will be convenient to denote the Lie derivative $-\l_\eta K$ ($\eta\in\vm(TM)$) by $[K,\eta]$. Then, for any vector field $\xi$ on $TM$,
\[[K,\eta]\xi=[K\xi,\eta]-K[\xi,\eta].\]
We have, in particular,
\begin{align}\eqnbetu{\lieJ}{a-c}
[\J,C]=\J;\quad [\J,X\vl]=0,\ [\J,X\tl]=0\quad (X\in \vm(M)).
\end{align}
\begin{quote}
\emph{In what follows, for simplicity, we shall denote also by $\i$, $\j$ and $\J$ the restrictions of these mappings to $\Sec(\splp)$ and $\vm(\splt M)$.}
\end{quote}
\end{szov}
\begin{szov}[Ehresmann connections]\
\ujsor (\textbf{a}) By an \emph{Ehresmann connection} in $\splt M$ we mean a $C^\infty(\splt M)$-linear mapping
\[
\H:\Sec(\splp)\longrightarrow\vm(\splt M)
\]
such that
\[
\j\circ\H=1_{\Sec(\splp)}.
\]

We emphasize (cf.\ \ref{Canob}(d)) that the $C^\infty(\splt M)$-linearity of $\H$ makes it possible to interpret an Ehresmann connection as a strong bundle map
\[
\mathsf H \colon\splt M\times_MTM\longrightarrow T\splt M
\]
as follows:

For each $(u,v)\in \splt M\times_MTM$ there exists a section $\X\in\Sec(\splp)$ such that $\X(u)=(u,v)$. Let $\mathsf H_u(v):=\H(\X)(u)$.
Then $\mathsf H_u$ is well-defined and
\[
\H(\X)(u)=\mathsf H_u(\X(u))\textrm{ \emph{for all} }\X\in\Sec(\splp).
\]

Obviously, the mappings
\[
\mathsf H_u:\{u\}\times T_{\mathring \tau(u)}M\longrightarrow T_u\splt M,\ u\in \splt M
\]
are linear. Now we obtain the desired mapping $\mathsf H:\splt M\times_MTM\rightarrow T\splt M$ by setting
\[
\mathsf H\upharpoonright\{u\}\times T_{\mathring\tau(u)}M:=\mathsf H_u.
\]
\ujsor (\textbf{b}) Let $\H:\Sec(\splp)\rightarrow\vm(\splt M)$ be an Ehresmann connection in $\splt M$.
Then $\vm\hl(\splt M):=\im(\H)$ is a submodule of $\vm(\splt M)$, and we have the direct decomposition $\vm(\splt M)=\vm\vl(\splt M)\oplus\vm\hl(\splt M)$.
Vector fields on $\splt M$ belonging to $\vm\hl(\splt M)$ are called \emph{horizontal}.
Notice that they do not form, in general, a subalgebra of the Lie algebra $\vm(\splt M)$. The mappings
\begin{align*}
\h:&=\H\circ\j,\ \v:=1_{\vm(\splt M)}-\h,\\
\V:&=\i\inv\circ\v\colon\ \vm(\splt M)\longrightarrow \Sec(\splp)
\end{align*}
are called the \emph{horizontal projection}, the \emph{vertical projection} and the \emph{vertical mapping} associated to $\H$, respectively. $\h$ and $\v$ are indeed projection operators in $\vm(\splt M)$, while the mapping $\V$ has the properties
\[\V\circ\i=1_{\Sec(\splp)},\ \kr(\V)=\im(\H).\]
The \emph{horizontal lift} of a vector field $X$ on $M$ (with respect to $\H$) is
\[X\hl:=\H(\kal X)=\h(X\tl).\]
($\kal X$ and $X\tl$ are regarded here as a section in $\Sec(\splp)$ and a vector field on $\splt M$, resp.; for simplicity, we make no notational distinction.)

\ujsor (\textbf{c}) An Ehresmann connection $\H$ is said to be \emph{homogeneous} if
\[[C,X\hl]=0\textrm{ \emph{for all} } X\in\vm(M).\]
Then $\H$, as a strong bundle map of $\splt M\times_MTM$ to $T\splt M$, may be extended continuously to a mapping $TM\times_MTM\rightarrow TTM$ such that
\[
\H(0_p,v)=(o_*)_p(v)\textrm{ \emph {for all} } p\in M,\ v\in T_pM.
\]
Thus, in what follows, we shall always assume that a homogeneous Ehresmann connection is defined
on the entire $TM\times_MTM$ (or on $\Sec(\pi)$).

\ujsor (\textbf{d}) If $\H$ is an Ehresmann connection in $\splt M$, then the mapping
\[\nabla:\vm(\splt M)\times\Sec(\splp)\longrightarrow\Sec(\splp),\ (\xi,\Y)\longmapsto \nabla_\xi\Y\]
given by
\begin{subequations}        
\begin{align}\label{Dv}
\nabla_{\v\xi}\Y:&=\nabla\vl_{\V\xi}\Y\mathop{=}^{\eqref{DvY}}\j[\v\xi,\H\Y]\\\label{Dh}
\nabla_{\h\xi}\Y:&=\nabla\hl_{\j\xi}\Y:=\V[\h\xi,\i\Y]
\end{align}
\end{subequations}
is a covariant derivative operator in the vector bundle $\splp$, called the \emph{Berwald derivative} induced by $\H$.

By the \emph{tension} of $\H$ we mean the $\nabla\hl$-differential $\mathbf t:=\nabla\hl\delta$ of the canonical section. Then, for any section $\X\in\Sec(\splp)$,
\begin{equation}\label{tens}
\mathbf t(\X):=(\nabla\hl\delta)(\X):=\nabla\hl_\X\delta=\V[\H\X,C].
\end{equation}
In particular,
\[\i\mathbf t(\kal X)=[X\hl,C],\quad X\in\vm(M);\]
therefore \emph{$\H$ is homogeneous if, and only if, its tension vanishes}.

With the help of the induced Berwald derivative we define the \emph{torsion} $\mathbf T$ of an Ehresmann connection $\H$ by
\[\mathbf T(\X,\Y):=\nabla_{\H\X}\Y-\nabla_{\H\Y}\X-\j[\H\X,\H\Y];\quad \X,\Y\in\Sec(\splp).\]
Evaluating on basic sections, we obtain the more expressive formula
\[\i\mathbf T(\kal X,\kal Y)=[X\hl, Y\vl]-[Y\hl,X\vl]-[X,Y]\vl;\quad X,Y\in\vm(M).\]
\end{szov}
\section{Lie derivative along the tangent bundle projection}
Let $\xi$ be a projectable vector field on $TM$ (\ref{generalities}(c)). We define a Lie derivative operator $\la_\xi$ on the tensor
algebra of the $C^\infty(TM)$-module $\Sec(\pi)$ by the rules
\begin{subequations}
\begin{align}
 \la_\xi \varphi:&=\xi\varphi, \textrm{ if } \varphi\in C^\infty(TM); \\
 \la_\xi \Y:&=\i\inv[\xi,\i\Y], \textrm{ if } \Y\in \Sec(\pi),
\end{align}
 \end{subequations}
and by extending it to the whole tensor algebra in such a way that $\la_\xi$ satisfies the product rule of tensor derivations. Since $\xi$ is a
projectable and $\i\Y$ is a vertical vector field, it follows that the vector field $[\xi,\i\Y]$ is vertical, so $\la_\xi\Y$ is well-defined.
If $\v=\i\circ\V$ is the vertical projection associated to an Ehresmann connection $\H$ in $TM$, then
$
 \i\inv[\xi,\i\Y]=\V[\xi,\i\Y],
$
so we get the useful formula
\begin{equation}\label{liealt}
 \la_\xi\Y=\V[\xi,\i\Y].
\end{equation}
Notice, however, that the Lie derivative operator $\la_\xi$ does not depend on any Ehresmann connection in $TM$.

If, in particular, $\xi:=X\tl$ or $\xi:=X\hl$, where $X$ is a vector field on $M$, then \eqref{liealt} takes the form
\begin{align}
\la_{X\tl}\Y=\V[X\tl,\i\Y],\label{lieteljes}
\end{align}
resp.
\begin{align}
\la_{X\hl}\Y=\V[X\hl,\i\Y]\mathop{=}^{\eqref{Dh}}\nabla\hl_{\kal{X}}\Y.
\end{align}
Since $[X\tl,\i\delta]=[X\tl,C]\stackrel{\textup{(\lieC b)}}{=}0$, it follows that
\begin{equation}\label{lxdelta}
 \la_{X\tl}\delta=0.
\end{equation}
The Lie derivative of a basic section with respect to a complete lift leads essentially to the ordinary Lie derivative. Namely, for any
vector fields $X$, $Y$ on $M$ we have
\[
\La \kal{Y}\mathop{=}^{\eqref{lieteljes}}\V[X\tl,Y\vl]\mathop{=}^{\textup{(\lielift b)}}\V[X,Y]\vl=\V\circ \i \kal{[X,Y]}=
\kal{[X,Y]}=\kal{\l_X Y}.
\]
This relation indicates that our Lie derivative operator $\La$ is a natural extension of the classical Lie derivative $\l_X$ on $M$.

\begin{lem}\label{lie4}
 For any projectable vector fields $\xi$, $\eta$ on $TM$,
\begin{equation}\label{liezj}
 [\la_\xi,\la_\eta]=\la_{[\xi,\eta]}.
\end{equation}
\end{lem}

\begin{proof}
 Obviously, both sides of \eqref{liezj} act in the same way on smooth functions on $TM$.
If $\Y$ is a section of $\pi$, then, applying \eqref{liealt} repeatedly,
\begin{align*}
[&\la_\xi,\la_\eta]\Y=\la_\xi \V[\eta,\i\Y]-\la_\eta \V[\xi,\i\Y]=\V ([\xi,\i\V[\eta, \i\Y]]-[\eta,\i\V[\xi,\i\Y]])
\\
&=\V ([\xi,[\eta, \i\Y]]+[\eta,[\i\Y,\xi]])=-\V[\i\Y,[\xi,\eta]]=\V[[\xi,\eta], \i\Y]=\la_{[\xi,\eta]}\Y.
\end{align*}
\end{proof}

\begin{lem}
 Let $X\in\vm(M)$, $\eta\in \vm(TM)$. Then
\begin{equation}\label{felcs}
 \La \j\eta=\j \L \eta.
\end{equation}
\end{lem}

\begin{proof}
 Since
\[
 0\mathop{=}^{\textup{(\lieJ c)}}[\J,X\tl]\eta=[\J\eta,X\tl]-\J[\eta,X\tl],
\]
we find
\[
 \i\La \j\eta=[X\tl,\J\eta]=\J[X\tl,\eta]=\i(\j\L \eta),
\]
which implies \eqref{felcs}.
\end{proof}

We end this section with the definition of the Lie derivative $\la_\xi D$ of a covariant derivative
$D\colon \vm(TM)\times \Sec(\pi) \to \Sec(\pi)$: it is given by the rule
\[
 (\la_\xi D)(\eta,\Z):=\la_\xi (D_\eta \Z)-D_\eta (\la_\xi \Z)-D_{[\xi,\eta]}\Z,
\]
where $\eta \in \vm(TM)$, $\Z\in \Sec(\pi)$.

Notice finally that the theory of Lie derivatives `along the tangent bundle projection' sketched here works without any change also on the
bundle $\splp\colon \mathring{T}M\times_M TM \to \mathring{T}M$.
\section{Affine vector fields on a spray manifold}
\begin{szov}
By a \emph{spray} for $M$ we mean a $C^1$ mapping $S:TM\to TTM$, smooth on $\overset\circ TM$, such that
   \begin{align}
   &\tau_{TM}\circ S=1_{TM}; \label{S1}
 \\
    &\mathbf{J}S=C;\label{S2}
  \\
     &[C,S]=S. \label{S3}
\end{align}

Condition (\ref{S2}) is equivalent to the requirement $\tau_*\circ S=1_{TM}$, so a spray for $M$ is a section also of the secondary vector bundle
$\tau_*\colon TTM\to TM$.
In view of \eqref{S3}, a spray is a \emph{homogeneous} vector field (of class $C^1$) \emph{of degree 2}.
We say that a manifold endowed with a spray is a \emph{spray manifold}.
\end{szov}
\begin{szov} If $\H$ is a homogeneous Ehresmann connection in $TM$, then $S:=\H\circ\delta$ is a spray for $M$, called the \emph{spray associated to} $\H$.
Indeed, for any vector $w$ in $TM$, $S(w)=\H(w,w)\in T_{w}TM$, therefore $\tau_{TM}(S(w))=w$, so \eqref{S1} is valid.
Since
\[
 \J\circ S=\i\circ \j\circ\H\circ\delta=\i\circ\delta=C,
\]
condition (\ref{S2}) also holds.
To check (\ref{S3}), observe first that the vector field $[C,S]-S$ is vertical, and hence $\mathbf{h}[C,S]=\h\, S$.
However, \mbox{$\mathbf{h} S=\H\circ\j\circ\H\circ\delta=\H\circ\delta=:S$,} so we get $\mathbf{h}[C,S]=S$.
On the other hand, by the homogeneity of $\H$,
\[
 0=-\i\mathbf{t}(\delta)=-\mathbf{v}[\H\circ\delta,C]=\mathbf{v}[C,S],
\]
therefore $\mathbf{h}[C,S]=[C,S]$ and $[C,S]=S$.
Finally, the $C^1$ differentiability of $S$ can be shown using the `Observation' in 3.11 (p. 1378) of \cite{setting}.

Thus sprays exist in abundance for a manifold.
Conversely, if $S$ is a spray for $M$, then there exists a unique torsion-free homogeneous Ehresmann connection $\H$ in $TM$ such
that the  horizontal lifts with respect to $\H$ are given by
\begin{equation}\label{c-g}
 X\hl:=\H(\widehat{X})=\frac{1}{2}(X\tl+[X\vl ,S]), \quad X\in\mathfrak{X}(M).
\end{equation}
For a proof of this fundamental fact we refer to \cite{setting}, 3.3, or to the original source \cite{crampin}.
The Ehresmann connection specified by \eqref{c-g} is said to be the \emph{Ehresmann connection induced by the spray $S$}.
\end{szov}
\begin{szov} Let $(M, S)$ be a spray manifold.
We say that a vector field  $X$ on $M$ is a \emph{projective vector field} for $(M,S)$ (or for the spray $S$) if there
is a continuous function $\varphi$ on $TM$, smooth on $\overset\circ TM$, such that
\begin{equation}\label{proj}
[X\tl,S]=\varphi\, C.
\end{equation}
If, in particular, $\varphi$ is the zero function, then we say that $X$ is an \emph{affine vector field} for $(M,S)$, or a \emph{Lie symmetry} of $S$.
\end{szov}
\begin{all}\label{affine}
Suppose $(M,S)$ is a spray manifold.
Let $\H$ be the Ehresmann connection induced by $S$, and let $\nabla$ be the Berwald derivative arising from $\H$.
For a vector field $X$ on $M$, the following  conditions are equivalent:
\begin{itemize}
 \item [(i)]$X$ is a Lie symmetry of $S$;
 \item  [(ii)]$[\mathbf{h},X\tl]=0$;
 \item[(iii)] $[\mathbf{v},X\tl]=0$;
 \item [(iv)] $\La\nabla=0$;
 \item[(v)] $[X\tl,Y\hl]=[X,Y]\hl$, for any vector field $Y$ on $M$;
\item[(vi)] $[\La ,\la_{Y\hl}]=\la_{[X,Y]\hl}$, $Y\in\mathfrak{X}(M)$;
\item [(vii)] $\La\circ\mathcal{V}=\mathcal{V}\circ\mathcal{L}_{X\tl}$.
\end{itemize}
\end{all}
\begin{proof}
The equivalence of conditions $(i)$, $(ii)$ and $(iv)$ has already been proved in \cite {lovas}.
\\
$(ii)\iff(iii)\quad$
This is evident, since $\v=\mathbf{1}-\mathbf{h}$ ($\mathbf{1}:=1_{\mathfrak{X}(TM)}$) and $[\mathbf{1},\xi]=0$ for all $\xi\in\mathfrak{X}(TM)$.
\\
$(ii)\iff(v)\quad$ For any vector field $Y$ on $M$,
\[
[\mathbf{h},X\tl]Y\tl=[\mathbf{h}Y\tl,X\tl]-\mathbf{h}[Y\tl,X\tl]=[Y\hl,X\tl]-\mathbf{h}[Y,X]\tl=[Y\hl,X\tl]-[Y,X]\hl,
\]
so the vanishing of $[\mathbf{h},X\tl]$ implies that $[X\tl,Y\hl]=[X,Y]\hl$.
The converse is also true, since $[\mathbf{h},X\tl]$ annihilates the module of vector fields:
for any vector field $\xi$ on $TM$ we have
\[
[\mathbf{h},X\tl]\mathbf{J}\xi=[\mathbf{h} \circ \mathbf{J}(\xi),X\tl]-\mathbf{h}[\mathbf{J}\xi,X\tl]=0.
\]
$(v)\iff(vi)\quad$
This is an immediate consequence of the identity
\[
[\La,\widetilde{\mathcal{L}}_{Y\hl}]=\widetilde{\mathcal{L}}_{[X\tl,Y\hl]}
\]
(see Lemma \ref{lie4}).
\\
$(iii)\iff (vii)\quad$ For any vector field $\xi$ on $TM$,
\[
\i\La(\V\xi)=[X\tl,\v\xi],\quad  \i\V(\L\xi)=\v[X\tl,\xi],
\]
hence
$\La(\V\xi)=\V(\L\xi)$ if, and only if,
\[
0=[\mathbf{v}\xi,X\tl]- \mathbf{v}[\xi,X\tl]=[\mathbf{v}, X\tl]\xi.
\]
\end{proof}
\section{Conformal vector fields on a Finsler manifold}
 \begin{szov} Let $(M,F)$ be a \emph{Finsler manifold}. We recall that the \emph{Finsler function} $F\colon TM\to \R$ here is assumed to be \emph{smooth} on $\overset\circ TM$,
 \emph{positive} \mbox{($F(v)>0,$} if $v \in \overset\circ TM$), \emph{positive-homogeneous of degree 1}
($F(\lambda\, v)=\lambda \, F(v)$ for all $v\in TM$ and positive real number $\lambda$), and it is also required that
the \emph{metric tensor}
\[
 g:=\frac{1}{2}\nabla\vl\nabla\vl F^2
\]
is \emph{fibrewise non-degenerate}.
The function $E:=\frac{1}{2}F^2$ is the \emph{energy function} of $(M,F)$.
The homogeneity of $F$ implies that over $\overset\circ TM$ we have
\[
 CF=F,\qquad CE=2E.
\]
The  \emph{Hilbert 1-form} of $(M,F)$ is
\begin{eqnarray*}
&&\widetilde{\theta}:=\nabla\vl E =F\nabla\vl F \mbox { -- in the pull-back formalism, }
\\
&& \theta:=d_\J E\mbox{ -- in the } \tau_{TM}\mbox{ formalism}.
\end{eqnarray*}
It is easy to check that
\[
 \widetilde{\theta}(\widetilde{X})=g(\widetilde{X},\delta) \mbox{ \emph{for each} } \widetilde{X}\in \Sec(\overset\circ{\pi}).
\]
$\widetilde{\theta}$ and $\theta$ are related by
\begin{equation}\label{theta}
\theta=\widetilde{\theta}\circ\j.
\end{equation}
\\
The 2-form
\[
\omega:=d\theta=dd_\J E
\]
on $\overset\circ{T}M$ is said to be the \emph{fundamental 2-form} of $(M,F)$.
Its relation to the metric tensor is given by
\begin{equation}\label{omegag}
 \omega(\J\xi,\eta)=g(\j\xi,\j\eta); \qquad \xi,\eta \in \mathfrak{X}(\overset\circ{T}M).
\end{equation}
The non-degeneracy of $g$ implies the non-degeneracy of $\omega$ -- and vice versa.
\end{szov}
\begin{lem}
 With the notations introduced above, let $(M,F)$ be a Finsler manifold, and let $X$ be a vector field on $M$. Then
\begin{align}\label{lxtheta}
&(\La\widetilde{\theta})\circ\j=\L \theta;
\\
&(\La g)(\j\xi,\j\eta)=(\L\omega)(\J\xi,\eta);\qquad \xi,\eta \in\mathfrak{X}(\overset\circ TM).\label{lxg}
\end{align}
\end{lem}
\begin{proof}
We check only the less trivial second relation:
\begin{eqnarray*}
 &&(\mathcal{L}_{X\tl}\omega)(\mathbf{J}\xi,\eta)
=X\tl\omega(\mathbf{J}\xi,\eta)-\omega(\L\mathbf{J}\xi,\eta)-\omega(\mathbf{J}\xi,\L \eta)
\\
&&\stackrel{\eqref{felcs},\, \eqref{omegag}}{=}
X\tl g(\mathbf{j}\xi,\mathbf{j}\eta)
-\omega(\L\mathbf{J}\xi,\eta)-g(\mathbf{j}\xi,\La\j\eta).
\end{eqnarray*}
Since $\L\J\xi=[X\tl,\J\xi]=-[\J,X\tl]\xi+\J[X\tl,\xi]=\J\L\xi$, the second term at the right-hand side of the above relation takes the form
\[
 \omega(\L\mathbf{J}\xi,\eta)=\omega(\J\L\xi,\eta)\stackrel{\eqref{omegag}}{=}g(\j\L\xi,\j\eta)\stackrel{\eqref{felcs}}{=}g(\La\j\xi,\j\eta).
\]
So we obtain
\[
 (\mathcal{L}_{X\tl}\omega)(\mathbf{J}\xi,\eta)=X\tl g(\mathbf{j}\xi,\mathbf{j}\eta)-g(\La\j\xi,\j\eta)-g(\mathbf{j}\xi,\La\j\eta)=(\La g)(\j\xi,\j\eta).
\]
\end{proof}
\begin{szov} We continue to assume that $(M,F)$ is a Finsler manifold.
The $2n$-form
\[
  \sigma:=\frac{(-1)^ {\frac{n(n-1)}{2}}}{n!}\omega^ n,
\]
where $\omega^ n=\omega\wedge...\wedge\omega$ ($n$ factors) is a volume form on $\overset\circ{T}M$, called the \emph{Dazord volume form} of $(M,F)$.
By the \emph{divergence} of a vector field $\xi$ on $\overset\circ{T}M$ (with respect to $\sigma$) we mean the unique function
$\dvr \xi\in C^{\infty}(\overset\circ TM)$ such that
\[
 \mathcal{L}_\xi\,\sigma=(\dvr \xi)\,\sigma.
\]
\end{szov}
\begin{lem}\label{divc}
If $(M,F)$ is a Finsler manifold, then the divergence of the Liouville vector field $C$ on $\overset\circ{T}M$ with respect to the Dazord volume form is $n=dim\, M$.
\end{lem}
\begin{proof}
$
 \mathcal{L}_C\omega=\mathcal{L}_C dd_\J E=d\mathcal{L}_C\,d_\J E\stackrel{\eqref{dJLC}}{=}dd_\J\,\mathcal{L}_CE-dd_\J E=2dd_\J E-dd_\J E=\omega.
$
From this it follows by induction that $\mathcal{L}_C \omega^n=n\,\omega^n$, whence our claim.
\end{proof}
\begin{szov}\label{finsler} If $(M,F)$ is a Finsler manifold, then there exists a unique spray $S$ for $M$ such that
\begin{equation}\label{spray}
 i_S\,dd_\J E=-dE\quad \mbox{ \emph{over} } \overset\circ{T}M,\mbox{ \emph{and} } S \upharpoonright o(M)=0.
\end{equation}
We say that $S$ is the \emph{canonical spray} of $(M,F)$; the Ehresmann connection induced by $S$ according to \eqref{c-g} is said to be
the \emph{canonical connection} of $(M,F)$.
It may be characterized as \emph{the unique torsion-free homogeneous Ehresmann connection $\H$ for $M$ which is compatible
with the Finsler function} in the sense that  $dF\circ\H=0$,
or, equivalently,
\[
X\hl F=0\quad \mbox{\emph{for all}}\quad X \in \mathfrak{X}(M).
\]
With the help of the canonical connection, we define the \emph{Sasaki extension} $G$ of the metric tensor $g$ of $(M,F)$ by the rule
\begin{equation}\label{sasaki}
 G(\xi,\eta):=g(\mathbf{j}\xi,\mathbf{j}\eta)+g(\mathcal{V}\xi,\mathcal{V}\eta); \qquad \xi,\eta \in \mathfrak{X}(\overset\circ TM),
\end{equation}
where $\V$ is the vertical mapping associated to $\H$.
Then $G$ is a Riemannian metric tensor on $\overset\circ{T}M$.

For subsequent applications, we collect here some further technical results.
\end{szov}
\begin{lem}
 For any section $\widetilde{X}$ in $\Sec(\pi)$, we have
\begin{equation}\label{nablavdelta}
 \nabla\vl_{\widetilde{X}}\delta=\widetilde{X}.
\end{equation}
\end{lem}
\begin{proof}
 Let $\H$ be a homogeneous Ehresmann connection for $M$ and let \mbox{$S:=\H\circ\delta$} be the spray associated to $\H$ (\textbf{3.2}). Then, applying the so-called Grifone identity (\cite{grifon}, Prop. I.7), we find that
\[
 \nabla\vl_{\widetilde{X}}\delta:=\j[\i\widetilde{X},\H\delta]=\j[\i\widetilde{X},S]=\widetilde{X}.
\]
\end{proof}
\begin{lem}
 The energy function of a Finsler manifold can be obtained from the metric tensor by
\begin{equation}\label{gdelta}
 g(\delta,\delta)=2E;
\end{equation}
from the fundamental 2-form by
\begin{equation}\label{omegacs}
 \omega(C,S)=2E,
\end{equation}
where $S$ is a spray for the base manifold.
\end{lem}
\begin{proof}
\ $g(\delta,\delta)=\nabla\vl(\nabla\vl E)(\delta,\delta)=\nabla\vl_\delta(\nabla\vl E)(\delta)=\nabla\vl_\delta(\nabla\vl E(\delta))-\nabla\vl E(\nabla\vl _\delta\delta)
\\
\vspace{0,5 cm} \stackrel{\eqref{nablavdelta}}{=}\nabla\vl_\delta (CE)-\nabla\vl E(\delta)=C(CE)-CE=4E-2E=2E;$
\\
$\omega(C,S)=dd_\J E(C,S)=C\, d_\J E(S)-S\, (d_\J E(C))-d_\J E([C,S])
=C(CE)-d_\J E(S)=4E-2E=2E.$
\end{proof}
\begin{lem} \label{divs}
The divergence of the canonical spray of a Finsler manifold vanishes.
\end{lem}
\begin{proof}\
$
 \mathcal{L}_S \omega=\mathcal{L}_Sdd_\J E\stackrel{\eqref{magic}}{=}i_Sddd_\J E+di_Sdd_\J E\stackrel{\eqref{spray}}{=}-ddE=0,
$
which
\\
implies our claim.
\end{proof}
\begin{szov} Let $(M,F)$ be a Finsler manifold.
We say that a vector field $X$ on $M$ is a \emph{projective}, resp.\ an \emph{affine vector field} of $(M,F)$, if it is a projective vector field, resp.\ a Lie symmetry for the canonical spray of $(M,F)$.
A vector field $X$ on $M$ is said to be a \emph{conformal vector field}, if the Lie derivative of the metric tensor of $(M,F)$ with respect to the
complete lift of $X$ satisfies the relation
\begin{equation}\label{lxcg}
 \La g=\varphi\, g
\end{equation}
for a continuous function $\varphi\colon TM\to\R$, of class $C^1$ on $\overset\circ TM$, called the \emph{conformal factor} of $X$.
Particular cases of conformal vector fields are \emph{homothetic vector fields}
for which the conformal factor is a constant function and \emph{isometric vector fields}, also called
\emph{Killing vector fields}, for which the conformal factor is the zero function on $TM$.
\end{szov}
\begin{lem}\label{conform}
 If $X$ is a conformal vector field on a Finsler manifold $(M,F)$ with conformal factor $\varphi$, then
 $
X\tl E=\varphi\, E.
 $
\end{lem}
\begin{proof}\
$2X\tl E\stackrel{\eqref{gdelta}}{=}X\tl (g(\delta,\delta))=(\La g)(\delta,\delta)+2\, g(\La \delta,\delta)\stackrel{\eqref{lxdelta}}{=}
(\La g)(\delta,\delta)
\\
\stackrel{\eqref{lxcg}}{=}\varphi\,g(\delta,\delta)\stackrel{\eqref{gdelta}}{=}2\,\varphi E.$
\end{proof}
\begin{lem}\label{verticallift}
If $X$ is a conformal vector field on a Finsler manifold $(M,F)$, then the conformal factor of $X$ is the vertical lift of a smooth function on $M$.
\end{lem}
\begin{proof}\
In view of the previous lemma, $X\tl E=\varphi\, E$, where
\\
\mbox{$\varphi\in C^0(TM)\cap C^1(\overset\circ TM)$}.
Acting on both sides of this relation by the Liouville vector field, we get
on the one hand
\[
C(X\tl E)=C(\varphi \, E)=(C\varphi)E+2\varphi E,
\]
on the other hand
\[
 C(X\tl E)=[C,X\tl]E+X\tl(CE)=2X\tl E=2\varphi E,
\]
so it follows that $(C\varphi)E=0$, and hence $C\, \varphi=0$.
This means that $\varphi$ is \mbox{positive-homogeneous} of degree 0, which implies (see, e.g., \cite{setting}, 2.6, \mbox{Lemma 2}) that $\varphi$ is of the form
$\varphi=f\circ\tau,\ f\in C^{\infty}(M)$.
\end{proof}
\begin{all}\label{conformal}
Let $(M,F)$ be a Finsler manifold. For a vector field $X$ on $M$,
the following conditions are equivalent:
\begin{itemize}
\item [(i)]$X$ is a conformal vector field with conformal factor $\varphi$;
\item [(ii)]$X\tl E=\varphi\, E$;
\item [(iii)]$\mathcal{L}_{X\tl}\theta=\varphi\,\theta$;
\item [(iv)]$\widetilde{\mathcal{L}}_{X\tl}\widetilde{\theta}=\varphi \,\widetilde{\theta}$;
\item [(v)]$\mathcal{L}_{X\tl}\omega=\varphi\,\omega+d\varphi \wedge d_\mathbf{J}E;\quad \varphi=f\circ\tau, \quad f\in C^{\infty}(M).$
\end{itemize}
\end{all}
\emph{In conditions} $(ii)-(iv)$, $\varphi \in C^0(TM)\cap C^1(\overset\circ TM)$.
\begin{proof}
The arrangement of our reasoning follows the scheme
\begin{center}\begin{tabular}{c@{ }c@{ }c@{ }c@{ }c@{ }}
$(i)$&$\Longrightarrow$&$(ii)$&&
\\
\rotatebox[origin=c]{90}{$\Longrightarrow$}&&\rotatebox[origin=c]{90}{$\Longleftarrow$}&&
\\
$(v)$& $\Longleftarrow$ &$(iii)$&$\iff$&$(iv).$
\end{tabular}\end{center}
$\ (i)\Longrightarrow\, (ii)\quad$ This is just a restatement of Lemma \ref{conform}.
\\
$(ii)\Longrightarrow (iii)\quad$ Let $Y$ be a vector field on $M$. We have on the one hand
\begin{eqnarray*}
 (\L \theta)(Y\vl)&=&X\tl(\theta(Y\vl))-\theta([X\tl,Y\vl])\stackrel{\textrm{(\lielift b)}}{=}X\tl(\theta(Y\vl))-\theta([X,Y]\vl)=0
\\
&=&(\varphi\,\theta)(Y\vl),
\end{eqnarray*}
since the vertical vector fields are annullated by the 1-form $\theta=d_\J E$. On the other hand,
\begin{align*}
(\L\theta)(Y\tl)
& \hspace{0.2 cm}=X\tl(d_\J E(Y\tl))-d_\J E([X\tl,Y\tl])
\stackrel{\textrm{(\lielift c)}}{=}X\tl(Y\vl E)-[X,Y]\vl E
\\
&\stackrel{\textrm{(\lielift b)}}{=}X\tl(Y\vl E)-[X\tl,Y\vl] E
=Y\vl(X\tl E)\stackrel{(ii)}{=}Y\vl(\varphi\,E)
\stackrel{(*)}{=} \varphi (Y\vl E)
\\
&\hspace{0.2 cm}=(\varphi\, d_\J E)(Y\tl)
=(\varphi\,\theta)(Y\tl).
\end{align*}
At step $(\ast)$ we used the fact that our condition $X\tl E=\varphi E$ implies, as it turns out from the proof of Lemma \ref{verticallift}, that $\varphi$ is a vertical lift.
Thus $\L\theta=\varphi\theta$, as we claimed.
\\
$(iii)\Longrightarrow (v)\quad$
 \vspace{-0.5 cm}
 \[\L\omega=\L d\,\theta=d\L\theta\stackrel{(iii)}{=}d(\varphi\,\theta)=d\,\varphi\wedge\theta +\varphi d\theta
=\varphi\,\omega+d\varphi\wedge d_\J E.\]
To check that the function $\varphi$ here is a vertical lift, we evaluate both sides of $(iii)$ at a spray $S$. Then
$\theta(S)=d_J E (S)=d\,E(C)=2E$, while
\[
 (\L\theta)(S)=X\tl(d_\J E(S))-d_\J E([X\tl, S])=2 X\tl E-\J[X\tl, S] E=2X\tl E,
\]
since $[X\tl, S]$ is vertical (see, e.g., \cite{setting}, p. 1350). Thus we obtain that $X\tl E=\varphi\,E$, which implies, as we have just remarked, that
$\varphi=f\circ\tau$, $f\in C^{\infty}(M).$
\\
$(v)\Longrightarrow (i)\quad$ For any vector fields $\xi,\eta$ on $\overset\circ TM,$
\begin{eqnarray*}
 (\La g)(\j\xi,\j\eta)&\stackrel{(\ref{lxg})}{=}&(\L\omega)(\J\xi,\eta)\stackrel{(v)}{=}(\varphi\,\omega+d\,\varphi\wedge d_\J E)(\J\xi, \eta)
\\
&=&\varphi\omega(\J\xi,\eta)+
d_\J\varphi(\xi)d_\J E(\eta)-d\varphi (\eta)d_\J E(\J\xi)
\\
&\stackrel{d_\J \varphi = 0}{=}&\varphi\omega(\J\xi, \eta)
\stackrel{(\ref{omegag})}{=}(\varphi\, g)(\j\xi,\j \eta),
\end{eqnarray*}
hence $\La g=\varphi g.$
\\
$(iii)\iff (iv)\quad$ If $\L\theta=\varphi\,\theta$, then for any vector field $\xi$ on $\overset\circ TM$,
\[
 (\La\widetilde{\theta})(\j\xi)\stackrel{\eqref{lxtheta}}{=}(\L\theta)(\xi)\stackrel{(iii)}{=} (\varphi\,\theta) (\xi)\stackrel{\eqref{theta}}{=}\varphi\,\widetilde{\theta}(\j\xi),
\]
whence $\La\widetilde{\theta}=\varphi\,\widetilde{\theta}$. The converse may be checked in the same way.
\end{proof}
We note that relation $(v)$, as a characterization of conformal vector fields
on a Finsler manifold, was announced first by J. Grifone \cite{grifon3}.
\begin{kov}
Let $(M,F)$ be a Finsler manifold. For a vector field $X$ on $M$, the following conditions are equivalent:
\begin{itemize}\label{homothetic}
\item [(i)] $X$ is a homothetic vector field, i.e., $\widetilde{\mathcal{L}}_{X\tl}g=\alpha\, g$, where $\alpha$ is a real number;
\item [(ii)] the energy function is an eigenfunction of $X\tl$ with eigenvalue $\alpha$, i.e., $X\tl E=\alpha \, E$;
\item [(iii)] $\L \theta=\alpha\,\theta$;
\item [(iv)] $\La\widetilde{\theta}=\alpha\,\widetilde{\theta}$;
\item [(v)] $\mathcal{L}_{X\tl}\omega=\alpha\,\omega$.
\end{itemize}
In conditions $(iii)-(v)$ $\alpha$ is a real number. With the choice $\alpha:=0$ we obtain criteria that a
vector field $X$ on $M$ be a Killing vector field of $(M,F)$.
\hspace{\stretch{5}} \qedsymbol
\end{kov}
\begin{all}\label{confsasaki}
Let $(M,F)$ be a Finsler manifold. If a vector field $X$ on $M$ is both affine and conformal, then
$X\tl$ is a conformal vector field on the Riemannian manifold \mbox{$(\overset\circ TM, G)$}, i.e.,
\mbox{$\L G=\varphi\, G$}, where \mbox{$\varphi \in C^0(TM)\cap C^1(\overset\circ TM)$} and $G$ is
the Sasaki extension of the metric tensor of $(M,F)$.

Conversely, if $X\tl$ is a conformal vector field of $(\overset\circ TM,G)$,
then $X$ is a conformal vector field on the Finsler manifold $(M,F)$.
\end{all}
\begin{proof}
 Suppose first that $X$ is both an affine and a conformal vector field on $(M,F)$.
Applying \eqref{sasaki}, \eqref{felcs} and Proposition \ref{affine}/$(vii)$,
for any vector fields $\xi,\eta$ on $\overset\circ TM$ we have
 \begin{eqnarray*}
  (\mathcal{L}_{X\tl}G)(\xi,\eta)
  &=&\mathcal{L}_{X\tl}(G(\xi,\eta))
  -G(\mathcal{L}_{X\tl}\xi,\eta)
  -G(\xi,\mathcal{L}_{X\tl}\eta)
  =\mathcal{L}_{X\tl}(g(\mathbf{j}\xi,\mathbf{j}\eta))
  \\
 &+&\mathcal{L}_{X\tl}(g(\mathcal{V}\xi,\mathcal{V}\eta))
-g(\mathbf{j}\mathcal{L}_{X\tl}\xi,\mathbf{j}\eta)
-g(\mathcal{V}\mathcal{L}_{X\tl}\xi,\mathcal{V}\eta)
\\
&-&g(\mathbf{j}\xi,\mathbf{j}\mathcal{L}_{X\tl}\eta)
-g(\mathcal{V}\xi,\mathcal{V}\mathcal{L}_{X\tl}\eta)=
\widetilde{\mathcal{L}}_{X\tl}(g(\mathbf{j}\xi,\mathbf{j}\eta))
\\
&+&\widetilde{\mathcal{L}}_{X\tl}(g(\mathcal{V}\xi,\mathcal{V}\eta))
-g(\widetilde{\mathcal{L}}_{X\tl}(\mathbf{j}\xi),\mathbf{j}\eta)
-g(\widetilde{\mathcal{L}}_{X\tl}(\mathcal{V}\xi),\mathcal{V}\eta)
\\
&-&g(\mathbf{j}\xi,\widetilde{\mathcal{L}}_{X\tl}(\mathbf{j}\eta))
-g(\mathcal{V}\xi,\widetilde{\mathcal{L}}_{X\tl}(\mathcal{V}\eta))
=(\widetilde{\mathcal{L}}_{X\tl}g)(\mathbf{j}\xi,\mathbf{j}\eta)
\\
&+&(\widetilde{\mathcal{L}}_{X\tl}g)(\mathcal{V}\xi,\mathcal{V}\eta)
=\varphi g(\mathbf{j}\xi,\mathbf{j}\eta)+ \varphi g(\mathcal{V}\xi,\mathcal{V}\eta)
=\varphi G(\xi,\eta).
\end{eqnarray*}
This proves that $X\tl$ is a conformal vector field on $(\overset\circ TM,G)$.
Conversely, under this condition we find that
\begin{eqnarray*}
  2\varphi\, E &=& \varphi\, g(\delta,\delta) = \varphi\, g (\mathcal{V}C,\mathcal{V}C)=\varphi\, G(C,C)=(\mathcal{L}_{X\tl}G)(C,C)
  \\
  &=&X\tl (G(C,C))-
G([X\tl ,C],C)-G(C,[X\tl ,C])=X\tl (G(C,C))
\\
&=&X\tl g(\delta,\delta)
=2X\tl E,
 \end{eqnarray*}
so, by Proposition \ref{conformal}, $X$ is a conformal vector field on $(M,F)$.
\end{proof}
\begin{all}\label{homaff}
Any homothetic vector field on a Finsler manifold is an affine vector field.
\end{all}
\begin{proof}
Let $(M,F)$ be a Finsler manifold, and let $S$ be the canonical spray for $(M,F)$.
Suppose that $X$ is a homothetic vector field of $(M,F)$.
Then, by Corollary \ref{homothetic}, there is a real number $\alpha$ such that $X\tl E=\alpha\, E$,
or, equivalently, $\L\,\omega=\alpha\,\omega$, so we have
\begin{eqnarray*}
\L d E&=&d(X\tl E)=\alpha\, d E
\stackrel{\eqref{spray}}{=}-\alpha\, i_S\, \omega
=-i_S(\alpha\,\omega)
=-i_S\,(\L\omega)
\\
&=&-\L i_S\,\omega+i_{[X\tl,S]}\omega
=\L dE+i_{[X\tl,S]}\omega.
\end{eqnarray*}
Thus $i_{[X\tl,S]}\omega=0$, and hence -- by the non-degeneracy of $\omega$ --
$[X\tl,S]=0$. This means that $X$ is a Lie symmetry of the canonical spray of $(M,F)$.
\end{proof}
\begin{lem}\label{divxc}
If $X$ is a conformal vector field on an $n$-dimensional Finsler manifold,
then (with respect to the Dazord volume form)
$
\dvr X\tl=n \, \varphi,
$
where $\varphi$ is the conformal factor of $X$.
\end{lem}
\begin{proof}
Choose a local frame $(X_i)_{i=1}^n$ for $TM$ over an open subset $U$ of $M$.
Then the family $(X_i\vl,X_i\tl)_{i=1}^n$
is a local frame for $TTM$ over $\tau^{-1}(U)$.
It may be shown by a little lengthy inductive argument that
\[
 (\L\omega)(X_1\vl,X_1\tl,...,X_n\vl,X_n\tl)=n\, \varphi\, \omega (X_1\vl,X_1\tl,...,X_n\vl,X_n\tl),
\]
which implies our claim.
\end{proof}
\begin{all}\label{projconf}
If a vector field is both a projective and a conformal vector field
on a Finsler manifold, then it is a homothetic vector field.
\end{all}
\begin{proof}
Let $(M, F)$ be an $n$-dimensional Finsler manifold. Suppose that a vector field $X$ on $M$ is both projective and conformal. Then, on the one hand,
\[
 [X\tl,S]=\psi\,C,\qquad \psi\in C^0(TM)\cap C^1(TM),
\]
where $S$ is the canonical spray of $(M,F)$. On the other hand, by \mbox{Proposition \ref{conformal}},
\[
X\tl E=f\vl E, \qquad  f\in C^{\infty}(M).
\]
Thus we get
\begin{eqnarray*}
2\psi E&=&\psi(CE)=[X\tl ,S]E=X\tl(SE)-S(X\tl E)=-S(f\vl E)
\\
&=&-(Sf\vl )E-f\vl(SE)
=-f\tl E,
\end{eqnarray*}
taking into account that $S$ is horizontal with respect to the canonical connection of $(M,F)$ and hence $SE=\frac{1}{2}SF^2=F(SF)=0$ (see \textbf{\ref{finsler}}),
applying furthermore the relation $Sf\vl=f\tl\ (f\in C^{\infty}(M))$, whose verification is routine. It follows that
\[
\psi=-\frac{1}{2} f\tl .
\]

Now we determine the divergence (with respect to the Dazord volume form) of both sides of the relation $[X\tl,S]= -\frac{1}{2}f\tl\, C$.
Applying the well-known rules for calculation (see, e.g., \cite{abraham}, \S 6.5 or \cite{lang}, XV,\S 1) we find that
\[
 \dvr [X\tl,S]=X\tl \dvr S - S \dvr X\tl \stackrel{\mbox{\begin{footnotesize}Lemmas\end{footnotesize}}\ \ref{divs},\,\ref{divxc}}{=} -S(nf\vl)
=-n f\tl
\]
and
\[
\dvr (-\frac{1}{2}f\tl C )=-\frac{1}{2}(Cf\tl+f\tl\dvr C)\stackrel{\mbox{\begin{footnotesize}Lemma\end{footnotesize}}\ \ref{divc}}{=}
-\frac{1}{2}(n+1)f\tl.
\]
So $(n-1)f\tl=0$, where $n\geq 2$ (\textbf{\ref{manifold}} (a)), whence $f\tl=0$. This implies by the connectedness of $M$ that $f$ is a constant function,
and therefore the conformal factor of $X$ is constant.
\end{proof}
We note that this result is an infinitesimal version of  Theorem 2 in \cite {szilasivince}.
\begin{all}\label{vol}
 Let (M,F) be a Finsler manifold. Suppose that a vector field $X$ on $M$ preserves the Dazord volume form of $(M,F)$, i.e., $\L\sigma=0$. If, in addition,
\begin{itemize}
 \item [(i)]  $X$ is a projective vector field, then $X$ is affine;
\item [(ii)] $X$ is a conformal vector field, then $X$ is isometric.
\end{itemize}
\end{all}
\begin{proof}
 First we note that our condition $\L\sigma=0$ implies that $\dvr X\tl=0$.
 \ujsor$(i)$ Suppose that $X$ is also a projective vector field, i.e.,
\[
 [X\tl,S]=\psi\,C, \qquad \psi\in C^0(TM)\cap C^1(\overset\circ TM).
\]
Observe that over $\overset\circ TM$ the function $\psi$ satisfies the relation $C\,\psi=\psi$. Indeed, by the Jacobi identity
\[
 0=[C,[X\tl,S]]+[X\tl,[S,C]]+[S,[C,X\tl]]=[C,[X\tl,S]]-[X\tl,S],
\]
hence
\[
 [X\tl,S]=[C,[X\tl,S]]=[C,\psi\,C]=(C\psi)C,
\]
therefore $(C\psi)C=\psi\,C$, and so $C\,\psi=\psi$.

Now, as in the previous proof, we calculate the divergence of both sides of the relation $[X\tl,S]=\psi\,C$. Since
$\dvr X\tl=\dvr S=0$, we have
\[
 \dvr[X \tl ,S]=X\tl \dvr S -S \dvr X\tl =0.
\]
On the other hand, by our above remark,
\[
 \dvr (\psi\,C)=\psi \dvr C+C\,\psi=(n+1)\psi.
\]
So it follows that $\psi=0$, hence $[X\tl,S]=0$. Thus $X$ is an affine vector field on $(M,F)$.
\ujsor $(ii)$ Now suppose that ($\dvr X\tl=0$ and) $X$ is also a conformal vector field. Then, by Proposition \ref{conformal},
$X\tl E=f\vl E,\ f\in C^{\infty}(M)$. Since
\[
 n\,f\vl\stackrel{\mbox{\begin{footnotesize}Lemma \end{footnotesize}}\ref{divxc}}{=}\dvr X\tl\stackrel{\mbox{\begin{footnotesize}cond.\end{footnotesize}}}{=}0,
\]
 it follows that $X\tl E=0.$ Thus, by Corollary \ref{homothetic}, $X$ is an isometric vector field on $(M,F)$.
\end{proof}
\section*{Acknowledgements}
This research was carried out in the framework of the Cooperation of the Czech and Hungarian Government (\dots).
The first author was supported also by Hungarian
Scientific Research Fund OTKA No.\ NK 81402. The authors wish to express their gratitude to Bernadett Aradi, Dávid Cs. Kertész and Rezs\oo o L. Lovas for their useful comments and
technical help during the preparation of the manuscript.

\vspace{2 cm}
József Szilasi 
\\
Institute of Mathematics, University of Debrecen
\\
H-4010 Debrecen, P.\ O.\ Box 12, Hungary
\\
\emph{E-mail: }szilasi@math.science.unideb.hu
\vspace{1 cm}
\\
Anna Tóth
\\
Institute of Mathematics, University of Debrecen
\\
H-4010 Debrecen, P.\ O.\ Box 12, Hungary
\\
\emph{E-mail: }tothanna@math.science.unideb.hu
\end{document}